\documentclass[12pt]{amsart}   % For Latex2e
\usepackage[english]{babel}
\usepackage{mathtools}
\usepackage{amssymb,amscd, graphics,color,latexsym,cancel}   % For Latex2e
\usepackage[linktoc=all, pagebackref, hyperindex]{hyperref}%
\hypersetup{colorlinks,  citecolor=blue,  filecolor=blue,  linkcolor=blue,  urlcolor=black}
\usepackage{tikz-cd} 
\usepackage{extpfeil}
\usepackage{tikz}
\usetikzlibrary{matrix,calc}
\usepackage{mathrsfs}
\usepackage[all]{xy}
\usepackage{amsfonts}
\usepackage[normalem]{ulem}
\usepackage{euscript}
\usepackage{amsmath}
\usepackage{ifpdf}
\usepackage{cleveref}
\LARGE\textwidth=6in
\newcommand{\rar}{\rightarrow}
\newcommand{\lar}{\longrightarrow}
\newcommand{\surjects}{\twoheadrightarrow}

\newtheorem{Theorem}{Theorem}[section]
\newtheorem{Lemma}[Theorem]{Lemma}
\newtheorem{Corollary}[Theorem]{Corollary}
\newtheorem{Proposition}[Theorem]{Proposition}

\theoremstyle{definition}
\newtheorem{Remark}[Theorem]{Remark}
\newtheorem{Example}[Theorem]{Example}
\newtheorem{Conjecture}[Theorem]{Conjecture}
\newtheorem{Definition}[Theorem]{Definition}
\newtheorem{Question}[Theorem]{Question}
\def\sqr#1#2{{\vcenter{\hrule height.#2pt
			\hbox{\vrule width.#2pt height#1pt \kern#1pt
				\vrule width.#2pt}
			\hrule height.#2pt}}}
\def\phi{\varphi}

\def\VaVa{{\mathcal V}\kern-5pt {\mathcal V}}
\def\gr#1#2{{\rm gr}\, _{#1}(#2)}
\def\gr{{\rm gr}\,}
\def\hht{{\rm ht}\,}
\def\depth{{\rm depth}\,}

\def\Min{{\rm Min}\,}
\def\codim{{\rm codim}\,}

\def\grade{{\rm grade}\,}
\def\rk{\rm rank}
\def\hom{\mbox{\rm Hom}}
\def\syz{\mbox{\rm Syz}}

\def\Ext#1#2#3#4{{\rm Ext}\,^{#1}_{#2}({#3},{#4})}

\def\proj#1{{\rm Proj}\, (#1)}
\def\supp#1{{\rm Supp}\, (#1)}

\def\ini{\mbox{\rm in}}

\def\der#1{\mbox{\rm der}_k(#1)}

\def\O{{\mathcal O}}

\def\cl#1{{\mathcal #1}}

\def\phi{\varphi}

\def\hht{{\rm ht}\,}
\def\grade{{\rm grade}\,}

\def\fm{{\mathfrak m}}

\def\NN{\mathbb N}

\def\fp{{\mathfrak p}}

\def\fm{{\mathfrak m}}

\def\NN{\mathbb N}

\def\cl#1{{\cal #1}}
\def\rk{\rm rank}

\newcommand{\excise}[1]{}
\def\NZQ{\mathbb}               % the font for N,Z,Q,R,C
\def\NN{{\NZQ N}}

\def\AA{{\NZQ A}}
\def\PP{{\NZQ P}}

\def\Jc{{\mathcal J}}

\def\G{{\mathcal G}}

\def\opn#1#2{\def#1{\operatorname{#2}}} % to make operators
\opn\chara{char} \opn\length{\ell} \opn\pd{pd} \opn\rk{rk}
\opn\projdim{proj\,dim} \opn\injdim{inj\,dim} \opn\rank{rank}
\opn\depth{depth} \opn\grade{grade} \opn\height{height}
\opn\embdim{emb\,dim} \opn\codim{codim}

\opn\Tr{Tr} \opn\bigrank{big\,rank}
\opn\superheight{superheight}\opn\lcm{lcm}
\opn\trdeg{tr\,deg}%\emph{
	\opn\reg{reg} \opn\lreg{lreg} \opn\ini{in} \opn\lpd{lpd}
	\opn\size{size} \opn\sdepth{sdepth}
	\opn\link{link}\opn\fdepth{fdepth}\opn\lex{lex}
	\opn\tr{tr}
	\opn\type{type}
	\opn\div{div} \opn\Div{Div} \opn\cl{cl} \opn\Cl{Cl}
	\opn\Spec{Spec} \opn\Supp{Supp} \opn\supp{supp} \opn\Sing{Sing}
	\opn\Ass{Ass} \opn\Min{Min}\opn\Mon{Mon}
	\opn\Ho{H}
	\opn\Ann{Ann} \opn\Rad{Rad} \opn\Soc{Soc}  \opn\der{Der}
	\opn\bour{Bour} 
	\opn\Im{Im} \opn\Ker{Ker} \opn\Coker{Coker} \opn\Am{Am}
	\opn\Hom{Hom} \opn\Tor{Tor} \opn\Ext{Ext} \opn\End{End}
	\opn\Aut{Aut} \opn\id{id}
	
	\opn\nat{nat}
	\opn\pff{pf}%   \pf exists already
	\opn\Pf{Pf} \opn\GL{GL} \opn\SL{SL} \opn\mod{mod} \opn\ord{ord}
	\opn\Gin{Gin} \opn\Hilb{HP}\opn\sort{sort}
	\opn\PF{PF}\opn\Ap{Ap}
	\opn\aff{aff} \opn
	\con{conv} \opn\relint{relint} \opn\st{st}
	\opn\lk{lk} \opn\cn{cn} \opn\core{core} \opn\vol{vol}  \opn\inp{inp} \opn\nilpot{nilpot}
	\opn\link{link} \opn\star{star}\opn\lex{lex}\opn\set{set}
	\opn\width{wd}
	\opn\Fr{F}
	\opn\QF{QF}
	\opn\G{G}
	\opn\type{type}\opn\res{res}
	\opn\log{Log}
	\opn\gr{gr}   
	
	\def\pot#1#2{#1[\kern-0.28ex[#2]\kern-0.28ex]}

	\opn\dirlim{\underrightarrow{\lim}}
	\opn\inivlim{\underleftarrow{\lim}}
\begin{document}
		%\begin{titlepage}
		
		\title[The Bourbaki degree ]{The Bourbaki degree of plane projective curves}
		\author{Marcos Jardim}
		\address{Universidade Estadual de Campinas (UNICAMP) \\ Instituto de Matem\'atica, Estat\'{\i}stica e Computação Cient\'{\i}fica (IMECC) \\ Departamento de Matem\'atica \\
			Rua S\'ergio Buarque de Holanda, 651\\ 13083-970 Campinas-SP, Brazil}
		\email{jardim@unicamp.br}
		\author{Abbas Nasrollah Nejad}
		\address{Department of Mathematics, Institute for Advanced Studies in Basic Sciences (IASBS), Zanjan 45137-66731, Iran}
		\email{abbasnn@iasbs.ac.ir}
		\author{Aron Simis}
		\address{Departamento de Matemática, CCEN, Universidade Federal de Pernambuco, Av. Jornalista Aníbal Fernandes, s/n, Cidade Universitária
			CEP 50740-560, Recife, Pernambuco
			Brazil}
		\email{aron.simis@ufpe.br}
		
		

		\begin{abstract}
			Bourbaki sequences and Bourbaki ideals have been studied by several authors since its inception sixty years ago circa. Generic Bourbaki sequences have been thoroughly examined by the senior author with B. Ulrich and W. Vasconcelos, but due to their nature, no numerical invariant was immediately available. Recently,  J. Herzog, S. Kumashiro, and D. Stamate introduced the {\em Bourbaki number} in the category of graded modules as the shifted degree of a Bourbaki ideal corresponding to submodules generated in degree at least the maximal degree of a minimal generator of the given module. The present work introduces the{\em Bourbaki degree} as the algebraic multiplicity of a Bourbaki ideal  corresponding to choices of minimal generators of minimal degree.
			The main intent is a study of plane curve singularities via this new numerical invariant. Accordingly, quite naturally, the focus is on the case where the standing graded module is the first syzygy module of the gradient ideal of a reduced form $f\in k[x,y,z]$ -- i.e., the main component of the module of logarithmic derivations of the corresponding curve. The overall goal of this project is to allow for a facet of classification of projective plane curves based on the behavior of this new numerical invariant, with emphasis on results about its lower and upper bounds.
		\end{abstract}
		\maketitle
		
		%	\tableofcontents
		
		\section*{Introduction}
		
		Much of  recent work on the homological facets of a plane curve $X=V(f)$ has largely focused on the nature of the gradient ideal $J_f$ of $f$. And yet, this approach is not so new (see, e.g., \cite{Simis-Koszul}, \cite{Simis-triply}, and possibly earlier), coming from the need to characterize certain classes of singularities in purely algebraic terms, that is to say, by drawing upon the core of commutative algebra.
		For example, a typical algebraic puzzle asks if there is a clear-cut watershed between the nature of an  ideal $I\subset R$ generated by three forms of the same degree and the gradient ideal of a reduced form. Of course, thanks to Dolgachev's theorem, the Cremona issue makes them as apart as possible. But, what about the homological behavior? Does it disentangle the conundrum as to when the strong geometric background of the gradient comes from a loose choice of three forms?
		
		The commutative algebra comes to help facing this conundrum by bringing out the related homology. This way, one trades out the hard geometric classification of singularity types for a ``classification'' of homological types referring to Betti numbers and degree shifts in minimal free resolutions. In terms of a classical algebraic language, one is after syzygies of graded modules and their degrees. In order to look closer at these syzygies one may resort to several methods, one of which is to reduce the problem to the case of homogeneous ideals.
		The choice of trailing after the idea of the Bourbaki ideal goes along this pattern.
		
		The notion of a Bourbaki sequence, that basically relates a module to an ideal via a pivotal submodule, was introduced in \cite{Bourbaki-source}.
		It has soon been considered by several authors, for various purposes (\cite{Ausl}, \cite{EvansGraham}, \cite{Miller}, \cite{Herzog-Kuhl},  \cite{RAM1}, \cite{Weyman}, \cite{Dimca-Sticlaru3}, \cite{Herzog_et_al}).
		The present work is closer to \cite{Herzog_et_al} in that one searches to extract a numerical invariant out of a Bourbaki sequence in the graded module category.
		However, in contrast to {\em loc.cit.}, where the Bourbaki number of a graded module $M$ is the shift of the Bourbaki ideal in a graded Bourbaki sequence corresponding to the choice of an equigenerated submodule generated in degree at least the degree of a minimal generator of $M$ of {\em maximal} degree, here we consider instead  a Bourbaki ideal coming from a Bourbaki sequence corresponding to choices of minimal generators of {\em minimal} possible degree.
		Moreover, the numerical invariant considered here is totally different, namely, it is the degree (multiplicity) of the Bourbaki ideal in the discussion.
		
		The above  is as much nearness as there exists between the two approaches.
		We focus on a particular graded module over the polynomial ring $R:=k[x,y,z]$, having in mind the understanding of projective plane curve theory from this angle.
		Namely, we consider the main direct summand of the module of logarithmic derivations associated with a reduced form $f\in R$. This module is canonically identified with the first syzygy module $\syz(R/J_f)$ of the gradient (Jacobian) ideal $J_f\subset R$ of $f$.
		Since $R/J_f$ determines the singularities of the associated projective plane curve $X:=V(f)\subset \mathbb P_k^2$, the quest has both algebraic and geometric facets.
		
		The novelty introduced here is a {\em Bourbaki degree} $\bour(X)$ of $X$, defined as the degree of $R/I_{\epsilon}$, where $I_{\epsilon}$ is the Bourbaki ideal associated to a choice of a minimal generator $\epsilon$ of $\syz(R/J_f)$ of minimal possible degree.
		Bourbaki ideals defined this way have been considered before by Dimca--Sticlaru (\cite[Section 5]{Dimca-Sticlaru3}, \cite[Section 3]{Dimca-Sticlaru4}) and, in a non-explicit way, by du Plessis--Wall (\cite{CTC-Plessis}).
		Both are closely related  to this work  but do not single out any new numerical invariant attached to a Bourbaki sequence, by rather focusing on other invariants, such as the total Tjurina number and the stability threshold related to the former.
		Drawing upon the facilitation of a Bourbaki degree, our work entails further precision to the relationship among some of these invariants and the curve geometry.
		We now briefly describe the sections of the paper.
		
		The first section contains notation and preliminaries for the rest of the paper.
		Here we introduce the module $\syz(R/J_f)$ of first syzygies of the gradient ideal $J_f$ in terms of the module of logarithmic derivations of $f\subset R=k[x,y,z]$, thereafter elaborating on its basic properties. It will be our standing graded module throughout the paper.
		
		The second section evolves around the details of a Bourbaki ideal $I_{\epsilon}\subset R$ associated with the choice of a minimal generator $\epsilon$ of $\syz(R/J_f)$. We give a formula for the degree (multiplicity) of $R/I_{\epsilon}$ in terms of the degree of $R/J_f$, the standard degree $e$ of $\epsilon$ and the homogeneous degree of $f$.
		This approach is pretty much the same as the one by Dimca and co-authors, but it has the advantage for the reader that it collects the basic facts and the respective proofs in one single theorem in purely algebraic terms, while the former is somewhat spread out in a few different papers and uses both algebraic and geometric arguments.
		
		Likewise, we recover the definitions introduced in du Plessis--Wall, where they make use of a suitable colon ideal.
		
		The tight relationship between the various terms in the formula above of the degree of $R/I_{\epsilon}$ has encouraged us to call it the {\em Bourbaki degree} $\bour(X)$
		of $X=V(f)\subset \PP^2_k$ associated to $\epsilon$, in the case where the homogeneous degree of the latter is the initial degree of $\syz(R/J_f)$.
		Focusing on the initial degree complies with Dimca's approach, although a Bourbaki degree as such has not been a theme thereof. 
		
		As a follow-up to the basic structural theorem and the introduction of the Bourbaki degree $\bour(X)$, we give some examples to illustrate the behavior of the latter. One of these has to do with the maximality behavior of nodal curves in this respect: besides attaining almost maximal bound for $\bour(X)$, it also frequently attains the maximal possible degree of a generating syzygy. In this thread of line, we conjecture that the curve $f=(x^2-y^2)z^{d-1} -(x^{d-1}-y^{d-1})x^2 -y^{d+1}$ (for any $d\geq 2$) does attain such maximum.
		
		Next, is one main theorem proving the upper bound $\bour(X)\leq e^2$ for a singular curve $X$, where $e$ is a minimal syzygy homogeneous degree. 
		Taking into account that the total Tjurina number coincides with the multiplicity $\deg(R/J_f)$, this bound implies the main result in \cite[Theorem 3.2]{CTC-Plessis}.
		The theorem also implies that if $e=1$ then $f$ is either free or nearly free, a result originally due to Dimca and co-authors.
		
		A note of interest in the proof of this theorem is the inclusion of radical ideals $\sqrt{H}\subset \sqrt{I_{\epsilon}}$, where $H$ denotes the ideal generated by the coordinates of a syzygy $\epsilon$ of minimal homogeneous degree. This inclusion might be strict or an equality as the choice of $\epsilon$ varies.
		
		In the sequel, we establish a few scattered facts in the case where $f$ is irreducible.
		Using that the total Milnor number is bounded by $d(d-1)$ in this case, we improve the structural lower bound to $\bour(X)$ all the way to $\bour(X)\geq d+e(e-d)$.
		As a corollary, there are no irreducible free divisors of degree less than $5$, a result that has been spread around as a clear statement, but no rigorous proof has been given so far. Finally, along this line, we show that if $f$ is an irreducible free divisor of degree $5$ then its minimal syzygy degree is $2$, a situation that repeats itself throughout many examples in the literature (see, e.g., \cite{SimToh}, \cite{Nanduri}).
		
		The section ends with a characterization of curves such that $\bour(X)\leq 2$, which takes care of the nearly free curves and some one-plus curves introduced by Dimca and co-authors.
		
		The last section focuses on lower bounds to $\bour(X)$.
		We restrict ourselves to quasi-homogeneous singularities, with particular endeavor for the case of nodal singularities. In the spirit mentioned before in this Introduction, yet another paradigm of the maximality phenomenon of nodal singularities now takes place, this time around  in regard to the minimal syzygy degree, which is now at least $d-1$.
		By drawing upon the known upper bounds for the number of singular points, we state lower bounds for $\bour(X)$ in the reduced and irreducible cases, reflecting the cases where $e=d-1$ or $e=d$.
		Next is a final blow about the maximality of nodal curves, to the effect that a curve having a unique singular point, which is nodal, is characterized by the value $\bour(X)=d^2-1$.
		
		We end the section with some comments on the relation to Dolgachev's theorem.
		In particular, we give another proof of the case of Dolgachev's theorem where one assumes that the curve is quasi-homogeneous.
		This observation has been stated before in \cite{Abbas} and \cite{Sim_on_Dolg}.
		
		\subsection*{Acknowledgments}
		MJ is supported by the CNPQ grant number 305601/2022-9 and the FAPESP Thematic Project number 2018/21391-1. AAN is partially supported by the FAPESP grant number 2022/09853-5. AS is partially supported by a CNPq grant number 301131/2019-8. AAN and AS thank for the warm hospitality during their visits to IMECC-UNICAMP.

		\section{Preliminaries}
		
		\subsection{Logarithmic derivations}	
		
		In this short section, we recall the notion of the module (respectively, sheaf) of logarithmic derivations, frequently referred to as the {\em Derlog} module (respectively,  sheaf). For an encompassing treatment, see \cite{Saito}, \cite{Simis_free}.
		
		Fix a perfect infinite field $k$ of high enough characteristic (as a function of certain data). For the geometric purpose, we may assume that it is algebraically closed of characteristic zero.
		Let $R=k[x_1, \ldots, x_m]$ denote a standard graded polynomial ring. The set of $k$-linear derivations of $R$ into itself is an $R$-module, denoted $\der_k(R,R)$, or more simply, $\der_k(R)$. The usual partial derivatives $\partial_{x_i}:=\partial/\partial x_i$, as maps from $R$ to $R$, are $k$-linear derivations and form a free basis of the $R$-module $\der_k(R)$.
		Since $R$ is standard graded, then $\der_k(R)$ admits a natural $\NN$-grading where 
		\[\der_k(R)_n=\sum_{i=1}^m R_n\partial_{x_i}, \; n\in \NN,\]
		where $R_n$ denotes the $k$-vector space spanned by the homogeneous polynomials of degree $n$.
		%Thus, a typical element of $\der_k(R)$ has the form $
		
		The following notion can be stated more generally for ideals $I\subset R$, but in this work, we are only interested in the case of a homogeneous principal ideal $(f)$, where $f\in R_{d+1}$, with $d\geq 1$. Thus, we are looking at a  hypersurface in projective space.
		
		\begin{Definition} 
			The {\em module of logarithmic derivations} of $f$ is the submodule
			\[\der_f(R):=\{\delta\in\der_k(R)\ | \ \delta(f)\in (f) \}\subset \der_k(R).  \]
		\end{Definition}
		Since $f\der_k(R)\subset \der_f(R)$,  latter has same rank as $\der_k(R)$
		
		It has yet another notable submodule, namely, 
		\[\der_f(R)^0:=\{\delta\in\der_k(R)\ | \ \delta(f)=0\}\subset \der_k(R).  \]
		
		Writing a $k$-linear derivation	as $\delta=\sum_{i=1}^m g_i\partial_{x_i}$ then, upon identification $\der_k(R)\simeq R^m$ via orderly mapping the partial derivatives to the canonical basis of $R^m$,  $\der_f(R)^0$ becomes the $R$-module $\mathcal{Z}(\partial)$ of (first) syzygies of the set $\{\partial f/\partial_{x_i}\, |\, 1\leq i\leq m\}$.
		Since the latter are canonical generators of the {\em gradient ideal} $J_f$ of $f$, we rename ${\rm Syz} (J_f):=\mathcal{Z}(\partial)$.
		
		With this notation, there is a remarkable short exact sequence of $R$-modules (see, e.g., \cite[Proposition 3.2]{Simis_free})
		$$0\rar {\rm Syz} (J_f) \lar \der_f(R) \lar J_f\colon \kern-2pt{_R} (f) \rar 0.$$
		Assuming, moreover, that char$(k)$ does not divide $d+1$, the above sequence splits, and the complementary direct summand is identified with the 
		Euler derivation $\delta_E:=\sum_{i=1}^m x_i\partial_{x_i}$: 
		\begin{equation}\label{decomposition}
			\der_f(R)={\rm Syz} (J_f) \bigoplus R\delta_E,
		\end{equation}
		
		\subsection{The role of syzygies}
		
		Note the graded presentation of the gradient ideal
		\begin{equation}\label{Gradient_presentation}
			0\rar  \sum_e R \epsilon_{d+e}(-(d+e)) \lar R^3(-d) \stackrel{\Theta_f}{\lar} J_f \rar 0,
		\end{equation}
		where $\Theta_f$ maps the canonical $i$th basis element to $\partial_{x_i}$, while $\epsilon_{d+e}$ is a syzygy of degree $d+e$ in the grading of $R^3(-d)$, i.e., of standard degree $e$ in the polynomial ring $R$. Shifting by $d$, we can write
		\begin{equation}\label{shifted_presentation}
			0\rar  {\rm Syz} (J_f) \lar R^3\stackrel{\Theta_f}{\lar} J_f (d)\rar 0.
		\end{equation}
		Thus,  throughout ${\rm Syz} (J_f)$ is the  module of syzygies of $J_f$ in a graded resolution of the latter, with degrees of generators  setback to standard degree.

		From now on, assume that $m=3$ and that $f$ is reduced (i.e., $(f)$ is a radical ideal). For common reasons, we replace the variables by $x,y,z$.
		Throughout,  $k$ is algebraically closed and char$(k)$ does not divide $d+1$. In particular, $f\in J_f$.
		
		It is natural to assume  that the partial derivatives $f_x,f_y,f_z$ of $f$ are algebraically independent over $k$, as otherwise in this dimension the Hesse--Gordan--Noether classical result forces $f_x,f_y,f_z$ to actually be $k$-linearly independent, i.e., dependent upon only two variables up to a linear change of variables (see, e.g., \cite[Proposition 2.7]{CilRuSim}). In other words, the algebraic dependence of $f_x,f_y,f_z$ implies that $V(f)$ is a cone.  The hypothesis that $f_x,f_y,f_z$ of $f$ are algebraically independent over $k$ can be rephrased, e.g., to the effect that $k$-subalgebra $k[f_x,f_y,f_z]\subset R=k[x,y,z]$ has dimension three, or that the gradient ideal $J_f$ has maximal analytic spread, or still that the fiber cone algebra of $J_f$ is (isomorphic to) $k[x,y,z]$.
		
		Since $f$ is reduced, $J_f$ has codimension at least two, hence it is locally Cohen--Macaulay in the punctured spectrum of $R=k[x,y,z]$ (being  primary to the maximal ideal or the ring itself in every localization).
		Thus, locally in the punctured spectrum, $J_f$ is either a codimension two perfect ideal or else the localized ring. 
		Then (\ref{shifted_presentation}) implies that locally in the punctured  spectrum, ${\rm Syz} (J_f)$ is free in two generators.
		In particular, ${\rm Syz} (J_f)$ has rank two. In addition, being a submodule of a free module, it is reflexive.
		
		Let $X:=V(f)\subseteq \PP^2$ denote the corresponding reduced  projective plane curve of degree $d+1$, with singular scheme ${\rm Proj}(R/J_f)$. Throughout we will be mainly focusing on a singular curve $X$, so ${\rm codim}(J_f)=2$ will be the case.
		
		The coherent sheaf associated to the graded $R$-module  ${\rm Syz} (J_f)$  is  denoted $\mathcal{T}_f$. With $\Jc_f$ denoting the ideal sheaf associated to the gradient ideal $J_f$, (\ref{shifted_presentation}) reads out as an exact sequence
		\begin{equation}\label{gradient_Geom}
			0\rar \mathcal{T}_f \lar \O_{\PP^2}\stackrel{\Theta_f}{\lar} \Jc_f(d)\rar 0.
		\end{equation}
		This gives an identification 
		\[{\rm Syz} (J_f)_n\simeq \mathrm{H}^0(\PP^2, \mathcal{T}_f(n)), \; n\in \NN, \]
		thus substantiating the terminology {\em vector fields sections} annihilating $f$ used in \cite{CTC-Plessis}.
		
		\section{The graded Bourbaki ideal associated with a projective plane curve}
		
		We introduce the Bourbaki ideal of a projective plane curve in an algebraic way, pretty much as in  \cite[Section 5]{Dimca-Sticlaru3}, with slight modifications to serve our present purpose.
		Generalizations to arbitrary dimensions can be found in \cite{Herzog-Kuhl} or \cite{RAM1} and, more recently, in \cite{Herzog_et_al}.
		
		\subsection{Basic theorem}
		Let $\epsilon\in {\rm Syz} (J_f)$ be a syzygy of degree $e$ (in the geometric language as above, a global section of the shifted sheaf $\mathcal{T}_f(e))$.
		The inclusion $\epsilon\in {\rm Syz} (J_f)$ induces a homogeneous injective map $R(-e) \stackrel{\epsilon}{\lar} {\rm Syz} (J_f)$ sitting on an exact sequence of graded $R$-modules
		\begin{equation}\label{basic_sequence}
			0\rar  R(-e) \stackrel{\epsilon}{\lar} {\rm Syz} (J_f) \lar {\rm coker}(\epsilon) \rar 0.
		\end{equation}
		Setting $M_\epsilon:=	{\rm coker}(\epsilon)$,  one has the following  commutative diagram of graded $R$-modules
		\begin{equation}\label{diag4}
			\begin{split} \xymatrix@R-2ex{ 
					& 0 \ar[d] & 0 \ar[d] & & \\
					& R(-e)\ar[d]^{{\epsilon}}\ar@{=}[r] & R(-e)\ar[d]^ -{\tilde{\epsilon}}& \\
					0\ar[r] & {\rm Syz} (J_f) \ar[r]\ar[d]^u & R^{3} \ar[r]^-{\Theta_f}\ar[d] & J_{f}(d) \ar@{=}[d]\ar[r] & 0 \\
					0\ar[r] & M_\epsilon \ar[r]\ar[d] & R^3/\tilde{\epsilon}(R(-e)) \ar[r]\ar[d] & J_{f}(d)  \ar[r] & 0\\
					& 0 & 0 &  
			} \end{split}
		\end{equation}
		The following is the basic theorem that will guide our results.
		It can be considered as an encore to \cite[Theorem 5.1]{Dimca-Sticlaru3}, with a different emphasis as convenient to the sequel.
		\begin{Theorem}\label{Bourbakiideal}
			Let $\epsilon\in {\rm Syz} (J_f)$ be a minimal homogeneous generator of degree $e\geq 1$.  Then the following hold. 
			\begin{enumerate}
				\item[{\rm (a)}]  $M_\epsilon$ is free if and only if $f$ is a free divisor.
				In this case, $\deg(R/J_f)=d^2 +e(e-d)$.
				
				\item[{\rm (b)}] 	If $f$ is not a free divisor, $M_\epsilon$ is  isomorphic to a proper homogeneous ideal $I_\epsilon\subset R$ of codimension two, such that the induced isomorphism $M_\epsilon\simeq I_\epsilon(e-d)$ is homogeneous of degree zero.
				In addition, for such an ideal $I_\epsilon$ one has 
				\begin{equation}\label{degree_of_Bourbaki_ideal}
					\deg (R/I_\epsilon)=\left\{
					\begin{array}{ll}
						d^2+e(e-d)-\deg(R/J_f) & \mbox{if $f$ is singular}\\
						d^2+e(e-d) & \mbox{if $f$ is smooth}.
					\end{array}
					\right.
				\end{equation}
				
				\item[{\rm (c)}] If $f$ is not a free divisor, pick a complete set of minimal homogeneous generators of the $R$-module ${\rm Syz} (J_f)$ that include $\epsilon$, and let
				\begin{equation}\label{resolution_of_syz}
					0\rar F_1 \stackrel{\left[\lambda\atop \psi\right]}{\lar}  R(-e)\oplus F_0 \stackrel{(\epsilon,\phi)}{\lar} 	{\rm Syz} (J_f)\rar 0 
				\end{equation}
				denote a minimal graded free resolution based on these generators, where the remaining shifts have not been specified.
				Then a graded minimal free resolution of $I_\epsilon$ is
				\begin{equation}\label{resolution_of_Bourbaki_ideal}
					0\rar F_1(d-e)) \stackrel{\psi}{\lar}  F_0(d-e)\stackrel{u\circ \phi}{\lar} I_\epsilon \rar 0.
				\end{equation}
				In particular, $I_\epsilon$  is a codimension $2$ perfect ideal.
			\end{enumerate}
		\end{Theorem}
		\begin{proof}
			First,  $M_\epsilon$ is a torsion-free module of rank one. Indeed, quite generally, if $F\subset E$ is a free submodule of a torsionfree module $E$ then $E/F$ is torsionfree. This is because $F:_E (a)=F$ for any $0 \neq a \in R$.
			Clearly, $\rk E/F=\rk E-\rk F$.
			
			(a)
			Suppose that $M_{\epsilon}$ is a free $R$-module. Since $M_{\epsilon}$ has rank one, $M_\epsilon\simeq R$. Then the exact sequence (\ref{basic_sequence}) splits
			as $\syz(J_f)\simeq R(-e)\oplus R$, that is, $\syz(J_f)$ is a free $R$-module generated by $\epsilon$ and another minimal generator, necessarily, of (standard) degree $d-e$. In other words, one gets a graded isomorphism 
			$$\syz(J_f)\simeq R(-e)\oplus R(d-e)=R(-e)\oplus R(-(e-d)).$$
			This means that $J_f$ has a Hilbert--Burch free resolution, i.e., it is a codimension two perfect ideal. In other words, $f$ is a free divisor (\cite[Proposition 3.7]{Simis_free}).
			
			The converse is clear.
			
			The second statement is a calculation with the Hilbert polynomial along the free resolution of $R/J_f$:
			$$0\rar R(-(d+e))\oplus R(-(2d-e))\lar R(-d)^3 \lar R \lar R/J_f\rar 0.$$
			
			(b) 
			Again quite more generally, any finitely generated torsionfree module over a Noetherian ring $R$, having  rank $r$, is isomorphic to a submodule of $R^r$ (\cite[Proposition 3.5.3]{Simis_book}).
			
			Thus, in the present case, $M_\epsilon$ can be embedded as an ideal $I_\epsilon\subset R$.
			Since $f$ is not a free divisor by assumption, item (a) implies that $I_\epsilon$ is not a free $R$-module, hence it is a proper ideal that is not principal.
			Suppose that $I_\epsilon$ has height one.
			Since $R$ is an UFD, $I_\epsilon=a\tilde{I}$, for some $0\neq a\in R$ and an ideal $\tilde{I}$ of height at least two. Then,  $I_\epsilon$ and $\tilde{I}$ are isomorphic as $R$-modules by multiplication by $a^{-1}$, hence $M_\epsilon\simeq  \tilde{I}$.
			Since $M_\epsilon$ is graded, $I_\epsilon$ is homogeneous and so $a$ and $\tilde{I}$ can be taken to be homogeneous.
			Finally, we rename $\tilde{I}$ back to $I_\epsilon$.
			
			Next, we claim that, under the assumption of the item,  $I_\epsilon$ has codimension exactly two.
			Indeed, if  $I_\epsilon))$ has a regular sequence of length three, then by ``d\'ecalage'' (\cite[Proposition 6.2.77]{Simis_book}) , $\Ext^1(I_\epsilon,R(-e))=\Ext^2(R/I_\epsilon,R(-e))=0$. Dualizing the exact sequence $0\rar R(-e) \lar {\rm Syz}(J_f) \lar I_\epsilon  \rar 0$ into $R(-e)$
			yields the split exact sequence
			$$0\rar \hom (I_\epsilon, R(-e)) \lar \hom ({\rm Syz}(J_f), R(-e)) \lar \hom (R(-e), R(-e))\simeq R \rar 0.$$
			Therefore, ${\rm Syz}(J_f)^*\simeq \hom ({\rm Syz}(J_f), R(-e))\simeq R^2,$
			hence ${\rm Syz}(J_f)$ is free since it is reflexive. Then $f$ is a free divisor, contradicting the assumption of this item.

			It remains to determine a shift $s$ such that the induced isomorphism $M_\epsilon\simeq I_\epsilon(s)$ is graded homogeneous (i.e., of degree zero).
			For this, we resort to a Hilbert function calculation.
			
			We use the additive property of the Hilbert polynomial on the lowest horizontal graded exact sequence in~\eqref{diag4}. Then we draw on the respective tautological presentations of $R(d)/J_f(d)$ and of $R(s)/I_\epsilon(s)$ and use the fact that, since both $R(d)/J_f(d)$ and $R(s)/I_\epsilon(s)$ are one dimensional, their degrees (multiplicities) coincide with the corresponding Hilbert polynomials.
			
			Thus, we get:		
			{\Small
				\begin{eqnarray*}
					0&=&\Hilb_{R^3}(t)-\Hilb_{R(-e)}(t)-\Hilb_{J_f(d)}(t)-\Hilb_{I_\epsilon(s)}(t)\\
					&=&\Hilb_{R^3}(t)-\Hilb_{R(-e)}(t)-\Hilb_{R(d)}(t)-\Hilb_{R(s)}(t)+ \Hilb_{R(d)/J_f(d)}(t) +\Hilb_{R(s)/I_\epsilon(s)}(t)\\
					&=&3\binom{t+2}{2}-\binom{t-e+2}{2}-\binom{t+d+2}{2}-\binom{t+s+2}{2} +\Hilb_{R(d)/J_f(d)}(t) +\mathfrak{e}_0(R/I_\epsilon)\\
					&=& (e-d-s)t+\frac{1}{2}(-e^2-d^2-s^2+3e-3d-3s) +\Hilb_{R(d)/J_f(d)}(t) +\mathfrak{e}_0(R/I_\epsilon),
			\end{eqnarray*}}
			
			\noindent because $\dim R/I_\epsilon=2$, where $\mathfrak{e}_0$ denotes the algebraic multiplicity of a homogeneous quotient $R/I$.
			Since we are thinking of $t$ as a variable over $k$, and $[\Hilb_{R(d)/J_f(d)}(t)]_t=0$ for $t\geq 1$ as $\dim R/J_f\leq 1$,  it follows that  $s=e-d$ and, additionally, 
			\begin{eqnarray*}
				\mathfrak{e}_0(R/I_\epsilon)&=&\frac{1}{2}\left( d^2+e^2+s^2-3(e-d-s) \right)- \Hilb_{R(d)/J_f(d)}(t)\\
				&=& d^2+e(e-d)- \Hilb_{R(d)/J_f(d)}(t).
			\end{eqnarray*}
			Now, replacing the multiplicity by the degree in case it is defined, we conclude:
			\begin{equation*}
				\deg (R/I_\epsilon)=\left\{
				\begin{array}{ll}
					d^2+e(e-d)-\deg(R/J_f) & \mbox{if $f$ is singular}\\
					d^2+e(e-d) & \mbox{if $f$ is smooth}.
				\end{array}
				\right.
			\end{equation*}
			%We note, for later purposes, that the latter values are kosher in the geometric version as well, where $R/I_\epsilon$ and $R/J_f$ are replaced by the respective projective varieties of $\PP^2$.
			
			(c)  With the assumed notation of the item, one has te following commutative diagram
			\begin{equation}\label{diagtorsion}
				\begin{split} \xymatrix@R-2ex{ 
						&  & 0\ar[d] &0\ar[d] & \\
						&0\ar[d]\ar[r]&R(-e)\ar[d]\ar@{=}[r]&R(-e)\ar[d]^{\epsilon}\ar[r]&0\\
						0\ar[r]&F_1\ar[d]\ar[r]&R(-e)\oplus F_0\ar[d]\ar[r]^-{(\epsilon,\, \varphi)}&\syz(J_f)\ar[d]\ar[r]&0\\
						0\ar[r]&G\ar[r]&(R(-e)\oplus F_0)/R(-e) \ar[d]\ar[r]^-{\overline{(\epsilon,\phi)}}&\syz(J_f)/R(-e)\ar[d]\ar[r]&0\\
						& & 0 & 0& 
				} \end{split}
			\end{equation}
			where $G$ is the kernel of $\overline{(\epsilon,\phi)}$.
			By the snake lemma, the leftmost vertical map $F_1\lar G$ is an isomorphism.
			Since the map $\overline{(\epsilon,\phi)}$ is identified with $u\circ \phi$, then the lowest horizontal exact sequence gives free  resolution of ${\rm Syz}(J_f) /R(-e)=M_{\epsilon}\simeq I_\epsilon(e-d)$.
			Shifting by $-(e-d)=d-e$ gives the stated free resolution of $I_{\epsilon}$..
			
			In particular, since $I_\epsilon$ has codimension two, it is identified with the ideal of maximal minors of a Hilbert--Burch matrix representing the rank two map $F_1\lar F_0$.
			Thus, it is a perfect (Cohen--Macaulay) ideal.
		\end{proof}		
		
		\begin{Remark}
			In wrapping up the proof of item (b) above, we set $\Hilb_{R(d)/J_f(d)}(t)=0$ when $J_f$ is an $\fm$-primary ideal. However, this is not saying that the algebraic multiplicity of the ring $R/J_f$ is zero, since, for an $\fm$-primary ideal $\mathfrak{a}\subset R$, the multiplicity of $R/\mathfrak{a}$ is defined as the length of $R/J_f$. In particular, this should not be confused with the geometric degree of ${\rm Proj}(R/J_f)$, which is zero because the latter is the empty set.
		\end{Remark}
		
		\begin{Corollary}
			If $f$ is not a free divisor, let $I_{\epsilon}$ stand by means of its  graded resolution  as in {\rm Theorem~\ref{Bourbakiideal} (c)}. Then $I_{\epsilon}=I_v$, where $I_v$ is the ideal of $R=k[x,y,z]$ as defined in \cite[After Lemma 1.1]{CTC-Plessis}.
		\end{Corollary}
		\begin{proof}
			This is an immediate consequence of the shape of $I_{\epsilon}$ via item (c) of Theorem~\ref{Bourbakiideal} and \cite[Proposition 2.1]{CTC-Plessis}.
		\end{proof}
		
		The exact sequence $0\rar R(-e) \lar {\rm Syz}(J_f) \lar I_\epsilon(e-d)  \rar 0$ is called a {\em Bourbaki sequence} of the curve $X=V(f)$, associated to the generating syzygy degree $e$.
		The ideal $I_{\epsilon}(e-d)$  is a {\it Bourbaki ideal} of of the curve $X=V(f)$. As such, it depends  on the choice of a minimal generator of a given degree.
		If $f$ is free, the obtained Bourbaki ideal is trivially $R(e-d)$, considered as an ideal.

		Recall that, for a graded module $M=\bigoplus_{i\geq 0}M_i$ over an $\NN$-graded Noetherian ring, its 
		{\it initial degree}  is defined to be $\mathrm{indeg}(M):=\min \{ i\ | \ M_i\neq 0\}.$
		Here the focus is on the case where $R$ is the standard graded polynomial ring $k[x,y,z]$

		\begin{Definition}\label{Defining_Bourbaki_degree}
			%	Let $e:=\mathrm{indeg}({\rm Syz}(J_f))$ in the standard grading.
			The {\em Bourbaki degree} of $f$ is the co-degree of the Bourbaki ideal of ${\rm Syz}(J_f)$ with respect to a minimal generator $\epsilon$ of standard initial degree.
			In other words, it is the degree (multiplicity) of the $R$-module $R/I_{\epsilon}$.
		\end{Definition}
		
		Note that the Bourbaki degree does not depend on the choice of minimal generator chosen, as long as it is of initial degree.
		Because of this independence, we denote it by $\mathrm{Bour}(f)$.
		By setting $X=V(f)$, the projective scheme $\mathrm{Proj}(R/I_\epsilon)$ is called a \textit{Bourbaki scheme} of $X$ and we will often write 
		$\mathrm{Bour}(f)=\mathrm{Bour}(X)$. 
		
		If $X$ is a free divisor, as has been seen, $I_\epsilon$ is identified with $R$. Thus, we set quite naturally   $\mathrm{Bour}(X)=0$.

		\begin{Remark}
			(1)  Theorem~\ref{Bourbakiideal} implies the following formula regardless for any reduced singular curve $X=V(f)\subset \PP^2$:
			\begin{equation}\label{formula_definite}
				\bour(X)=
				d^2+e(e-d)-\deg(R/J_f) 
			\end{equation}
			Geometrically, this formula is still kosher when $X$ is smooth, provided one thinks of $\deg(R/J_f) $ as the geometric degree of ${\rm Proj}(R/J_f)$, which is zero.

			(2)  Supposing that $V(f)$ is not smooth, let $J_f=J_f^{\rm un}\cap \mathfrak{M}$ denote a primary decomposition of the ideal $J_f$, where $J_f^{\rm un}$ is its minimal component (``un'' for {\em unmixed part}), while $\mathfrak{M}$ denotes an irrelevant component.
			Clearly,  $J_f^{\rm un}=J_f^{\rm sat}:=J_f\colon (x,y,z)^{\infty}$.
			Thus, $J_f^{\rm un}$ defines the singular locus of $V(f)$ scheme theoretically.
			Algebraically, since $\depth R/J_f^{\rm un}=1$, the ideal  $J_f^{\rm un}$ is Cohen--Macaulay, thus yielding yet another codimension two Cohen--Macaulay ideal. In particular,  by (\ref{formula_definite}), we are being told that the ubiquitous numerical expression $d^2+e(e-d)$ is the sum of the algebraic co-degrees of two such ideals.
			
			Setting $\fm=(x,y,z)$,  the cohomology module $H_{\fm}(R/J_f)=J_f^{\rm un}/J_f$ is of finite length, so eventually $[J_f^{\rm un}]_l=[J_f]_l$ for $l >> 0$.
			The least such $l$  is the {\em saturation number} of $J_f$, denoted ${\rm sat}(J_f)$.
			A nontrivial numerical relation has been established in \cite[Theorem 1.2]{Hamid-A}, to the effect that ${\rm indeg}(H_{\fm}(R/J_f))=3(d-3)- {\rm sat}(J_f)$.
			The algebraic--combinatorial side of ${\rm indeg}(H_{\fm}(R/J_f))$ has been explored in \cite{Busito}, while the more geometrically inclined facet has been considered in \cite{Hamid-A} and \cite{SimToh}, concerning a watershed bound related to plane Cremona maps, to wit, ${\rm indeg}(H_{\fm}(R/J_f))\geq d+1$. 
			It would be interesting to understand the role of this lower bound in the present case of a gradient ideal of a plane curve.
		\end{Remark}

		\subsection{Propaedeutic examples}
		As a rule, when bringing up examples, we assume that the ground field characteristic is well-behaved with respect to the essential exponents.
		
		\begin{Example}
			Let  $f=(x^2-y^2)z^{d-1} -(x^{d-1}-y^{d-1})x^2 -y^{d+1}$ ($d\geq 2$).
			
			It is clear that the point $p=(0:0:1)$ is a singular point. The local expression of $f$ at $(x,y)$ tells it is a simple node. 
			Now, $f_z=(d-1)(x^2-y^2)z^{d-2}$. Therefore, any minimal prime $\wp$ of the gradient ideal $J_f$ must contain either  $x+y,x-y$, or $z$. If $z\in \wp$, then $\wp$ must contain the respective pure $\{x,y\}$-parts of both $f_x$ and $f_y$. A closer inspection shows that these have no common factors in $k[x,y]$. Therefore, $\hht \wp=3$, which is absurd.
			If, instead, say, $x-y\in \wp$, then again upon inspection $\wp$ contains the ideal $(x-y, y^d,zy^{d-1})$. Since the radical of the latter is $(x,y)$, then $\wp=(x,y)$.
			Wrapping up, we have argued that $p=(0:0:1)$ is the unique singular point of $X=V(f)$.
			Moreover,  locally at $(x,y)$, the gradient ideal is easily seen to be generated by $x,y$.
			Therefore, the associativity formula for the degree (alternatively, the Tjurina number (see Proposition~\ref{deg_of_gradient__is_Tjurina})) gives that $\deg (R/J_f)=1$.
			
			This example is treated in more generality in Proposition~\ref{sing_is_one_node} regarding the Bourbaki degree, and hence, the emphasis is on the standard initial degree of $\syz(R/J_f)$. The latter has been established in \cite[Example 2.2 (i)]{Dimca-Sernesi}. The present example emphasizes instead the maximal generating  degree of $\syz(R/J_f)$. In virtue of the basic assumption in the Bourbaki treatment in \cite{Herzog_et_al}, this may have interest. Concerning this, we pose the following questions:
			\begin{Conjecture}
				Let $f$ be as above.
				\begin{enumerate}
					\item The maximal degree of a minimal generator of $\syz(J_f)$ is $2d-2$ (maximal possible according to \cite[Corollary 11]{Choudary-Dimca}).
					\item The minimal graded free resolution of $J_f$ is of the form
					\begin{equation}
						0 \rar R(-(3d-1))^2 \lar R(-2d)^3\oplus R(-(3d-2))\lar R(-d)^3 \lar R,
					\end{equation}
					where the generating syzygies in degree $2d$ are the Koszul syzygies.
					In particular, the regularity of $R/J_f$ is $3d$.
				\end{enumerate}
			\end{Conjecture} 
			The degree of the Bourbaki ideal with respect to the syzygy of (standard) degree $2d-2$ is $3(d-1)^2$, much larger than the Bourbaki degree as introduced here, which is $d^2-1$ (they coincide if and only if $f$ is cubic).
			Concerning  the issue of an upper bound for degrees of minimal generating syzygies, 
			\cite[Proposition 6]{Lazard}, has a general bound for the degree of generating syzygies of homogeneous matrices -- in the case of equigenerated ideals in $k[x,y,z]$ it amounts to $3/2(d^2+d-2)$. In higher dimension an upper bound was established for the syzygies of the gradient ideal in \cite[Corollary 11]{Choudary-Dimca} in the case of isolated singularities.
			
			The conjecture has been verified affirmatively, with computer assistance, for all sufficiently large values of $d$.
			
		\end{Example}

		\begin{Example}
			Let $X\subseteq\PP_k^2$ be the quintic curve defined by the polynomial $f=x^5+x^4y+x^3z^2+y^2z^3$. A computation with \cite{DGPS} shows that $\syz (J_f)$  is minimally generated by two syzygies ${\epsilon}_1$, and ${\epsilon}_2$ of standard degree $3$ and two additional ones of standard degree $4$. Moreover,  $\syz(J_f)$ has the following minimal graded resolution
			\[0\rar R^2(-5)\stackrel{\begin{bmatrix}
					4y^2&3x^2-10yz\\
					-6z&-9y\\
					-x&-2z\\
					-3z^2&-10xz-8yz
			\end{bmatrix}}{\lar} R(-3)\oplus (R(-3) \oplus R^2(-4))\rar \syz(J_f)\rar 0,\]	 
			where, say, $\epsilon_1$ is chosen as the first term in an ordered set of minimal generators of  $ \syz(J_f)$.
			By Theorem~\ref{Bourbakiideal}, the minimal graded resolution of the Bourbaki ideal $I_{{\epsilon}_1}$ is 
			\[0\rar R^2(-4)\stackrel{\begin{bmatrix}
					-6z&-9y\\
					-x&-2z\\
					3z^2&-10xz-8yz
			\end{bmatrix}}{\lar} R^2(-3)\oplus R(-2)\rar I_{{\epsilon}_1}\rar 0.\]
			By the above exact sequence,  the Hilbert Series of $R/I_{{\epsilon}_1}$ is 
			\[\mathrm{HS}_{R/I_{{\epsilon}_1}}(t)=\dfrac{1-t^2-2t^3+2t^4}{(1-t)^3}=\dfrac{1+2t+2t^2}{(1-t)}.\]
			Therefore, $\mathrm{Bour}(X)=1+2+2=5$ and $\deg(R/J_f)=8$.
			
			Note that, had we chosen $\epsilon_2$ instead, the resulting Bourbaki ideal $I_{\epsilon_2}$ would have the same Betti numbers and shifts, hence the same value for the Bourbaki degree.
		\end{Example}
		\begin{Example}
			Let $X=V(f)$ be a  sextic curve  defined by $f=xy^4z+x^6+y^6$.  A computation with \cite{DGPS} shows that $\syz(J_f)$ is generated by the column vectors of the matrix 
			\[\begin{bmatrix}
				0&       -6y^3z&      -6y^4\\
				xy&      9x^4+y^2z^2&   y^3z\\
				-6y^2-4xz&-54x^3y-4yz^3&36x^4-4y^2z^2
			\end{bmatrix}
			\]
			with graded minimal resolution
			\[0\rar R(-5)\stackrel{\begin{bmatrix}
					9x^3\\
					-y\\
					z
			\end{bmatrix}}{\lar}R(-2) \oplus R^2(-4)\rar \syz(J_f)\rar 0.\]
			Choosing $\epsilon$ as the unique column vector of degree $2$,
			by Theorem~\ref{Bourbakiideal} the minimal graded resolution of $I_{\epsilon}$ is 
			\[0\rar R(-2)\stackrel{\begin{bmatrix}
					z\\
					-y\\
			\end{bmatrix}}{\lar} R^2(-1)\rar I_{\epsilon}\rar 0. \]
			Therefore, the Bourbaki scheme $\proj{R/I_{\epsilon}}$ is a reduced and irreducible complete intersection of degree one, hence 
			$\mathrm{Bour}(X)=1$ and $\deg(R/J_f)=18$. 
			
			This is an example falling within a large class of curves introduced by Dimca and co-authors under the name {\em nearly free curves} (\cite{DimcaFreenessTjurina}) -- see Proposition~\ref{Bour_for_nearlyfree} for further details.
			It is worth noting that the Bourbaki scheme associated with the other syzygies is not even reduced and has degree $3$, pointing to the fact that the Bourbaki invariants considered in \cite{Herzog_et_al} may be different and possibly less tight.  
		\end{Example}
		
		\subsection{A distinguished upper bound}
		
		We give an upper bound for the Bourbaki degree solely in terms of the initial degree of the syzygies of the gradient ideal of a curve.

		\begin{Theorem}\label{loweBour-number}
			Let $X=V(f)$ be a singular curve of degree $d+1$ and let $e=\mathrm{indeg}(\syz(J_f))$. Then 
			\begin{equation}\label{bounds}
				\bour(X)\leq e^2.	
			\end{equation}
		\end{Theorem}	
		\begin{proof}

			Dualize into $R$ the lowest horizontal exact sequence in~\eqref{diag4}, with $M_{\epsilon}=I_{\epsilon}(e-d)$. Observing that $\Ext^2(R^3/\tilde{\epsilon}(R(-e)),\, R)=0$ because $R^3/\tilde{\epsilon}(R(-e))$ has homological dimension one, entails the short exact sequence 
			\begin{equation}\label{degreeExt1}
				\Ext^1(R^3/\tilde{\epsilon}(R(-e)),\, R)\lar\Ext^1(I_{\epsilon}(e-d),R)\lar \Ext^2(J_f,R)\rar 0.	
			\end{equation}
			%By Lemma~\ref{{degree_CI}}, one has $\Ext^1(I_{\epsilon}(e-d),R)\simeq 
			Since $J_f$ is Cohen--Macaulay locally on the punctured spectrum, $\Ext^2(J_f,R)$ vanishes locally on the punctured spectrum. 
			
			\smallskip
			
			{\bf Claim 1:} $\Ext^1(R^3/\tilde{\epsilon}(R(-e)),\, R)\simeq (R/H)(e)$, where $H\subset R$ is the ideal generated by the coordinates of $\epsilon$ regarded as column vector.
			
			\smallskip
			
			To see this, dualize into $R$ the mid vertical exact sequence in~\eqref{diag4} to get the exact sequence
			\[0\rar (R^3/\tilde{\epsilon}(R(-e)))^*\rar  R^3\stackrel{\tilde{\epsilon}^t}{\rar} R(e)\rar \Ext^1(R^3/\tilde{\epsilon}(R(-e)),\, R)\rar 0,\]
			where the image of the map  $\tilde{\epsilon}^t$ is $H(e)$, with $H$ as explained.  Thus, we obtain the following short exact sequence
			\[0\rar H(e) \rar R(e)\rar \Ext^1(R^3/\tilde{\epsilon}(R(-e)),\,R)\rar 0, \]
			giving $\Ext^1(R^3/\tilde{\epsilon}(R(-e)),\, R)\simeq (R/H)(e)$ canonically. 
			
			So much for the claim.
			We then have:
			
			\smallskip
			
			{\bf Claim 2:} Every minimal prime of $R/I_{\epsilon}$ is a minimal prime of $R/H$, and moreover, for any such prime one has a surjective map \begin{equation}\label{surjection}
				(R/H)(e)_{\wp} \surjects R/I_{\epsilon}(e-d)_{\wp}.
			\end{equation}

			\smallskip

			Indeed, substituting the result of Claim 1 in (\ref{degreeExt1}) and localizing in the punctured spectrum, we obtain surjective maps 
			$$	(R/H)(e)_{\wp} \surjects \Ext^1(I_{\epsilon}(e-d),R)_{\wp}$$
			throughout all primes $\wp\neq \fm$.
			In particular, this holds if $\wp$ is a minimal prime of $R/I_{\epsilon}$. In addition, it forces any such prime to be a minimal prime of $R/H$; in fact, it lies on the support of $R/I_{\epsilon}$, while similarly,  the support of $R/H$ in codimension two is the set of minimal primes of $R/H$.
			
			Next, recall that the module $\syz(R/J_f)$ is free of rank two locally on the punctured spectrum. Therefore, the Bourbaki sequence implies that  the Bourbaki ideal $I_{\epsilon}(e-d)$ is generically a complete intersection (of two generators).
			Then, as is well-known,  $\Ext^1(I_{\epsilon}(e-d),R)_{\wp}\simeq R/I_{\epsilon}(e-d)_{\wp}$ for every minimal prime of the latter.
			
			This takes care of Claim 2.
			
			Now, note the length inequality from (\ref{surjection}) 
			$$	\ell ((R/H)(e)_{\wp}) \geq  \ell (R/I_{\epsilon}(e-d)_{\wp}).$$
			As a result, we get $\deg(R/I_{\epsilon}(e-d)) \leq \deg (R/H)$ via
			the associativity formula for degrees.
			The stated bound now follows since $\bour(X)=\deg(R/I_{\epsilon}(e-d)) $ and $\deg(R/H)\leq e^2$ because $H$ is generated by three forms of degree $e$.
		\end{proof}

		%Thus, by the associativity formula, it obtains
		%\begin{eqnarray*}
		%\deg(R/I_{\epsilon}(e-d)) &=& \deg (\Ext^1(I_{\epsilon}(e-d),R)), \quad\text{\rm by %Lemma~\ref{degree_CI}}\\
		%	&=& \sum_{\fp \in \Min(R/I_{\epsilon})} \ell (\Ext^1(I_{\epsilon}(e-d),R)_{\fp}) \,% \deg (R/\fp) , \; \text{\rm by \cite[Lemma 1.2]{AndrSimVasc}}\\ [3pt]
		%	&\leq & \sum_{\fp \in \Min(R/I_{\epsilon})} \ell ((R/H)(e)_{\fp}) \, \deg (R/\fp), \; \text{\rm by (\ref{surjection})} \\
		%	&\leq & 	\sum_{\fp \in \Min(R/H)} \ell ((R/H)(e)_{\fp}) \, \deg (R/\fp),  \; \text{\rm by Lemma~\ref{minimal_primes}}\\
		%	&=& \deg (R/H).
		%\end{eqnarray*}

		\begin{Corollary}{\rm \cite[Corollary 1.4 (ii)]{DimcaFreenessTjurina}}
			Notation as above. If $e=1$ then the curve $X$ is either a free divisor or a nearly free curve.
		\end{Corollary}
		\begin{proof}
			By Theorem~\ref{loweBour-number}, $\bour(X)=0$ or $1$. As accounted for before, these values characterize the two types of curves in the statement.
		\end{proof}

		\begin{Example}
			Perhaps the simplest example where the above upper bound is attained is the higher cusp singularity defined by the polynomial $f=y^{d}z+x^{d+1}$. 
			One has $J_f=(x^d,y^{d-1}z, y^d)$.
			An application of the Buchsbaum--Eisenbud acyclicity criterion yields the minimal resolution
			\[ 0\rar R(-2d-1) \stackrel{\phi_2}{\lar} R(-d-1)\oplus R^2(-2d)\stackrel{\phi_1}{\lar} R^3(-d)\rar R\rar R/J_f\rar 0, \]
			where $\phi_1$ and $\phi_2$ are respectively defined by the matrices
			$$\begin{bmatrix}
				0&-y^d &-y^{d-1}z\\
				y& 0    & x^d\\
				z&x^d  &0
			\end{bmatrix} \: \text{\rm and}\: \begin{bmatrix}
				x^d\\
				z\\
				-y
			\end{bmatrix}.$$
			By Theorem~\ref{Bourbakiideal} (c), it follows that the Bourbaki ideal corresponding to the initial degree ($=1$) is generated by $y,z$, hence $\bour(X)=1$.
			We note that cusp singularities are a special case of  nearly free curves. In fact, as we will show in Proposition~\ref{Bour_for_nearlyfree}, the value $\bour(X)=1$ characterizes such curves.
			This example is actually a member of  the curve family studied in \cite[Example and Theorem 3.5]{Simis-triply}. In the latter the above bound is again attained for singular members and, moreover, the minimal syzygy degree $e$ is quite arbitrary.
		\end{Example}	
		
		\begin{Example}
			By Claim 2 in the proof of the above theorem, one has an inclusion of radicals $\sqrt{H}\subset \sqrt{I_{\epsilon}(e-d)}$. The nature of this inclusion is not independent of the choice of $\epsilon$. The following example illustrates its behavior: 
			$$f=x^2y^2+x^2z^2+y^2z^2+2xyz ((1/2)x+y+z).$$
			A run of \cite{DGPS} yields that $\syz(R/J_f)$ admits two independent minimal generators in the initial degree $2$ such that, for one of them, the above inclusion of radicals is strict, while for the other one, it is an equality.
		\end{Example}
		
		\subsection{Special results}
		In this part, we collect a few results for the Bourbaki degree of certain classes of curves.
		
		The following result characterizes smooth curves in terms of the Bourbaki degree. 
		\begin{Proposition}\label{smooth}
			Let $X$ be a reduced curve of degree $d+1\geq 2$. Then $\mathrm{Bour}(X)=d^2$ if and only if $X$ is smooth.  
		\end{Proposition}
		\begin{proof}
			If $X$ is smooth, then $J_f$ is generated by an $R$-sequence, hence ${\rm Syz}(f)$ is generated by the Koszul relations. Therefore, $e=d$.
			In this case, the result follows from   \eqref{degree_of_Bourbaki_ideal}.
			Conversely, let $\mathrm{Bour}(X)=d^2$ and assume that $X$ is singular.
			Then \eqref{degree_of_Bourbaki_ideal} yields $\deg(R/J_f)=e(e-d)$. Since $e$ is the initial  degree of ${\rm Syz}(J_f) $ then $e\leq d$, hence  $\deg(R/J_f)\leq 0$, which is impossible with $\dim R/J_f=1$.
		\end{proof}

		In the non-smooth case, if $f$ irreducible, one can add:
		\begin{Proposition}
			Let $f\in R=k[x,y,z]$ be a homogeneous polynomial of degree $d+1\geq 2$ and let $e:={\rm indeg}(\syz(R/J_f))$. Supposing that $f$ is irreducible, one has:
			\begin{enumerate}
				\item[\rm (i)] If $f$ is not a free divisor then $\bour(X)\geq e(e-d)+d$.
				\item[\rm (ii)] If $f$ is a free divisor then $e\geq 2$ and $d\geq 4;$ in particular, there are no irreducible free divisors of degree $2\leq d+1\leq 4$. Moreover, if $d+1=5$ then $e=2$.		
				\item[\rm (iii)] If $d\geq 2$ and $\bour(X)=1$ then $e<d$.
			\end{enumerate}
		\end{Proposition}
		\begin{proof}
			(i) We have $\deg (R/J_f)=\tau(X)\leq \mu(X)\leq d(d-1)$, where the last inequality is stated in  \cite[Chapter 4, Section 4, (4.5)]{DimcaBook1} when $f$ is irreducible.
			Now apply (\ref{degree_of_Bourbaki_ideal}).
			
			(ii) If $f$ is free then by Theorem~\ref{Bourbakiideal} (a), one has $\deg(R/J_f)=d^2+e(e-d)$ and, on the other hand, as above,  $\deg (R/J_f)\leq d(d-1)$. Therefore, $e\geq 2$.
			
			For the additional statement, one can verify that $d\leq 3$ leads to a contradiction with the inequality $e(e-d)+d\leq 0$.
			
			(iii)  Again, by the assumption and the above inequality, we have  $e(e-d)+d\leq 1$.
			Another  calculation with $e\geq d$ (hence, $e=d$) and $d\geq 2$ leads to a contradiction.
		\end{proof}

		\begin{Remark}
			Note that (ii) may fail if $f$ is not irreducible, e.g.,  $f=xyz$. 
		\end{Remark}	
		
		Recall the following class of curves introduced in  \cite{Dimca-Sticlaru2}. 
		\begin{Definition}
			A reduced curve $X=V(f)$ of degree $d+1$ is called {\em nearly free} if  $R/J_f$ has a graded minimal free resolution of the form
			\begin{equation}\label{Res-Nearlyfree}
				0\rar R(-d-a_2-1)\rar R(-d-a_1)\oplus R^2(-d-a_2)\rar R^3(-d)\rar R\rar R/J_f\rar 0,
			\end{equation}
			for some integers $1\leq a_1\leq a_2$. 
		\end{Definition}
		In analogy to the case of free curves, these shifts are called the {\em exponents} of $X$, and satisfy the equality $a_1+a_2=d+1$ -- in particular, one has $a_1\leq (d+1)/2$.
		
		The following result characterizes nearly free curves in terms of the Bourbaki degree. 
		\begin{Proposition}\label{Bour_for_nearlyfree}
			Let $X=V(f)$ be a reduced curve of degree $d+1$. Then $\bour(X)=1$ if and only if $X$ is nearly free. 
		\end{Proposition}
		\begin{proof}
			Assume that  $\bour(X)=1$. Since a Bourbaki ideal is a codimension two perfect ideal, the Bourbaki ideal $I_{\epsilon}$ of $\syz(J_f)$ associated to a syzygy $\epsilon$ of degree $e=\mathrm{indeg}(\syz(R/J_f))$ is a complete intersection generated by two linear forms and has the graded minimal free resolution
			\[0\rar R(-2)\rar R(-1)\oplus R(-1)\rar I_{\epsilon}\rar  0. \]
			Now we apply part (c) of  Theorem~\ref{Bourbakiideal}. One has 
			\[0\rar R((e-d)-2)\rar R^2((e-d)-1)\oplus R(-e)\rar \syz(R/J_f)\rar 0,\]
			which forces that $d-e+1\geq e$. Therefore, we get  
			\[0\rar R(-d-(d-e+1)-1)\rar  R(-d-e)\oplus  R^2(-d-(d-e+1)) \rar R^3(-d)\rar R\rar R/J_f\rar 0,\]
			which proves the assertion.
			
			Conversely, assume that $X$ is nearly free with exponent $a_1\leq a_2$. Then by~\eqref{Res-Nearlyfree}, the $R$-module $\syz(R/J_f)$ has the graded minimal resolution of the form 
			\[0\rar R(-a_2-1)\stackrel{\begin{bmatrix}
					h\\
					\ell_1\\
					\ell_2
			\end{bmatrix}}\rar R(-a_1)\oplus R^2(-a_2)\rar \syz(R/J_f) \rar 0,\]
			where $\ell_1, \ell_2$ are linear forms. 
			By Theorem~\ref{Bourbakiideal} (c), choosing the unique generator $\epsilon$ of standard degree $a_1$, we get the minimal graded resolution of $I_{\epsilon}$
			\[0\rar R(-a_2-(a_1-d)-1)\stackrel{\begin{bmatrix}
					\ell_1\\
					\ell_2
			\end{bmatrix}}{\rar }  R^2(-a_2-(a_1-d))\rar I_{\epsilon}\rar 0,\]
			which forces that $a_2+(a_1-d)=1$ as $I_{\epsilon}$ is a codimension $2$ perfect ideal.  
			Thus, $I_{\epsilon}$ is generated by two linear forms, hence $\bour(X)=\deg(R/I_{\epsilon})=1$.  
		\end{proof}
		Since the definition of nearly free curves is given in terms of the shape of the corresponding free resolution, one might expect that larger values of $\bour(X)$ are responsible for certain shapes of the resolution as for $\bour(X)=1$.
		This sort of approach has been the main thread of recent work of Dimca--Sticlaru, by stressing  numerical invariants other than the one here.
		The following proposition is an example of how one may proceed one step further with a case embodied in the event of $3$-{\em syzygy curves} as in \cite[Proposition 3.1]{Dimca-Sticlaru4}.
		
		\begin{Proposition}
			Let $X=V(f)$ be a singular singular curve of degree $d+1$ and let $e={\rm indeg}(\syz(J_f))$.  Then $\bour(X)=2$ if and only if  
			$R/J_f$ has a graded minimal free resolution of the form
			{\small
				\begin{eqnarray*}
					0\rar R(-(d+(d-e+3)))\kern-7pt&\rar & \kern-7pt R(-(d+(d-e+2)))\oplus R(-(d+(d-e+1)))\oplus R(-(d+e))\\ 
					&\stackrel{\phi}{\rar}&R^3(-d)\rar R\rar R/J_f \rar 0, 
				\end{eqnarray*}
			}
			where $e\leq d+1/2$.
		\end{Proposition}
		\begin{proof}
			Suppose that $R/J_f$ has the graded minimal resolution as stated. Then, $\syz(J_f)={\rm im}(\phi)$ has the graded minimal free resolution 
			\[0\rar R(-(d-e+3))\rar R(-(d-e+2))\oplus R(-(d-e+1))\oplus R(-e)\rar \syz(J_f)\rar 0. \]
			By  Theorem~\ref{Bourbakiideal} (c) the graded minimal resolution of $I_{\epsilon}$ is of the form 
			\begin{equation}\label{Bour2}
				0\rar R(-3)\rar R(-2)\oplus R(-1)\rar I_{\epsilon}\rar 0.
			\end{equation}
			Hence 
			\[\mathrm{HS}_{R/I_{{\epsilon}}}(t)=\dfrac{1-t-t^2+t^3}{(1-t)^3}=\dfrac{1+t}{(1-t)}.\]
			Therefore, $\bour(X)=\deg(R/I_{\epsilon})=2$. 
			Conversely, assume that the Bourbaki number of $X$ is two. In this degree the codimension two perfect ideal $I_{\epsilon}$ is necessarily a complete intersection generated by a linear form  and a quadratic form, which implies the graded minimal free resolution 
			\[0\rar R^3(-(d-e+3))\rar R(-(d-e+2))\oplus R(-(d-e+1))\rar I_{\epsilon}(e-d), \]
			which forces that $e\leq (d+1)/2$. By the proof of Theorem~\ref{Bourbakiideal} (c), it gives a minimal free resolution of $\syz(J_f)$, which extends to the stated resolution of  $R/J_f$ by shifting by $-d$.
		\end{proof}
		
		We note that a reduced curve whose gradient ideal admits a free resolution as in the above proposition is a special case of both a $3$-syzygy curve and a {\em one-plus curve} -- the latter having been characterized in terms of the shape of the free resolution of the corresponding gradient ideal in \cite[Theorem 2.3]{Dimca-Sticlaru4}, partially drawing on \cite{Hamid-A}.
		
		It is quite obvious that, in terms of the Bourbaki degree, the result of \cite[Proposition 3.1]{Dimca-Sticlaru4} reads to the effect that a $3$-syzygy curve is characterized by $\bour(X)$ being the product of two suitable integers $\geq 1$.
		It remains the question as to when such integers validate the effective existence of a corresponding $3$-syzygy curve. For the case where these are equal integers, the singular curves in the family in \cite[Example]{Simis-triply} respond to this question.
		
		\begin{Remark}
			The inverse question, as to whether or when a given homogeneous codimension two Cohen--Macaulay ideal $I\subset R$ is the Bourbaki ideal of some torsionfree maximal Cohen--Macaulay module $M$, has been solved in \cite{Herzog-Kuhl}, under the additional  requirement that  the sought Bourbaki sequence be of the form $0\rar F \rar M \rar I\rar 0$, where $F$ is free of rank equal to the {\em type} of $I$ (i.e., the number of columns in an $(n+1)\times n$ matrix defining $I$).
			In addition, the solution is essentially unique.
			
			To relate to our setup, we have to observe two points: first, in our typical Bourbaki sequence, the departing  module $\syz(R/J_f$ is torsionfree -- hence, has maximal dimension -- but it has depth $2$, so is not Cohen--Macaulay. Thus, an inverse question would have to change gear at this point, possibly renouncing to any uniqueness statement.
			Second, the free kernel in our Bourbaki sequence has rank one. Thus, to engage in the above, even forgetting about the Cohen--Macaulay property, we would impose that $I$ has type one, which means that $I$ is a complete intersection of two forms.
			
			Then, in this situation, pretty much like in the proof of Proposition~\ref{Bour_for_nearlyfree} (first direction), we can try to reconstruct a certain free resolution and expect it to be the resolution of a torsionfree $R$-module isomorphic to $\syz(R/J_f)$, for some $f\in R$.
			Though it sounds quite loose, this is what is happening in the proof of the theorem when the two forms are linear.
			
			Of course, in our setup, the natural question would be as to when a given homogeneous codimension two Cohen--Macaulay ideal $I\subset R$ is the Bourbaki ideal of some reduced curve $X$, for a suitable generating syzygy degree.
			At the moment, we even lack some natural necessary conditions.
		\end{Remark}

		\section{Lower bounds to the Bourbaki degree}
		
		For any two positive integers $d,e$, one has $d^2+e(e-d)>d^2-2ed+e^2= (d-e)^2\geq 0$. Therefore,  \eqref{degree_of_Bourbaki_ideal} implies that $d^2+e(e-d)$  is an upper bound for the degree of a Bourbaki ideal associated with any choice among the minimal generators of $\syz(J_f)$. In particular,  the Bourbaki degree of a plane curve of degree $d+1$ is bounded above by $d^2+e(e-d)$, where $e={\rm indeg}(\syz(J_f))$ -- note that this bound bounces between the two extremes $d^2$ and $d(d-1)+1$.
		
		In this section, we deal with finding a corresponding lower bound.
		
		\subsection{Milnor and Tjurina numbers}
		We provide a brief retrospection of the Milnor and Tjurina numbers in the case of a reduced singular curve $X=V(f)\subseteq \PP^2$ of degree $d+1\geq 2$.
		
		Let  $p\in \mathrm{Sing}(X)$. Viewing $p$ as a point of $\PP^2$, its ideal  is a minimal prime $\wp$ of the one-dimensional ring $R/J_f$.Thus, assuming that  $k$ is algebraically closed,  $\wp$ is generated by two linear forms. By a projective transformation,  assume, say, that $p=[0:0:1]$.  Considering the affine chart $U_z=\AA_k^2$ with coordinate ring $A=k[T,U]=k[x/z,y/z]$, $p$ corresponds to the point whose ideal is the maximal ideal $\fp:=(T,U)$. Set accordingly, $F(T,U):=f(x/z,y/z, 1)$, and let $J_F=(F_X,F_Y)$ be the gradient ideal of $F$ in $k[T,U]$ and $\tilde{J}_F:=(F,F_T,F_U)$ its ``full'' Jacobian ideal.
		
		\begin{Definition}\rm
			The {\em Milnor number} (respectively, the {\em Tjurina number}) $\mu_{\fp}(F)$   (respectively, $\tau_{\fp}(F)$) of $F\in k[T,U]$ at $\fp$ is $\dim _k A_{\fp}/(J_F)_{\fp}$ (respectively, $\dim _k A_{\fp}/(\tilde{J}_F)_{\fp}$.
			We often write $\mu_{\fp}(X)$   (respectively, $\tau_{\fp}(X)$) if no confusion arises.
		\end{Definition}
		We call $M_{\fp}(F):=A_{\fp}/(J_F)_{\fp}$ and $T_{\fp}(F):=A_{\fp}/(\tilde{J}_F)_{\fp}$ the Milnor local algebra and the Tjurina local algebra at $\fp$, respectively.
		The {\em total Milnor number} (respectively,  {\em total Tjurina number}) of $V(f)$ 
		is the summation of the local Milnor numbers (respectively, local Tjurina numbers) at all  singular points of $X=V(f)$.
		The latter will be denoted by $\mu(X)$ and $\tau(X)$, respectively.
		
		%	Clearly, $\tau(X)\leq\mu(X)$, while $\mu(X)\leq (d+1-1)^2=d^2$.  (\textcolor{blue}{(\sc How does one prove the latter inequality? Where is it applied anyway?)})
		
		The following is well-known, but we include a proof:
		\begin{Proposition}\label{Tjurina_is_deg_grad} {\rm ($k$ is algebraically closed and  char$(k)$ does not divide $d+1$)}. Let $f\in R$ be a reduced form of degree $d+1$ defining a singular curve. Then, with the above notation, one has
			\begin{equation}\label{deg_of_gradient__is_Tjurina}
				\deg(R/J_f)=\tau(X).
			\end{equation}\
		\end{Proposition}
		\begin{proof}
			By a linear change of coordinates, we may assume that $z\notin \wp$, for every minimal prime $\wp$ of $R/J_f$.
			%	The degree of the projective scheme $\mathrm{Proj}(R/J_f)$, denoted by  $\deg(R/J_f)$, is equal to the total Tjurina number $\tau(X)$. Indeed, up to the projective change of coordinates, we may assume that all singular points of $X$ all lie in one and the same affine piece. Using the above notation, we have  $\deg(R/J_f)$ is equal the length $\ell(A/J_F)$. By the well-known property of the Artinian ring, one has $\ell(A/J_F)=\sum_i\ell(A_{\fp_i}/(J_F)_{\fp_i})$, where $i$ runs through the set indexing the minimal primes of $A/J_F$ (i.e., the finitely many singular points of the affine curve).  
			By the associativity formula, one has:
			\begin{eqnarray*}
				\deg (R/J_f)&=&\sum_{\wp} \ell(R_{\wp}/\mathcal{P}_{\wp}), \; \text{\rm where $\mathcal{P}$ is the ${\wp}$-primary component}\\
				&=& \sum_{\wp} \ell(R_{\wp}/(J_f)_{\wp})=\sum_{\wp} \ell(R_{\wp}/(f_x,f_y,f_z)_{\wp})\\
				&=& \sum_{\wp} \ell(R_{\wp}/(f,f_x,f_y)_{\wp})\\
				&=& \tau(X)
			\end{eqnarray*}
			where $z\notin {\wp}$ and the Euler relation $(d+1)f=f=xf_x+yf_y+zf_z$ imply that $f_z\in (f,f_x,f_y)$.
		\end{proof}	
		
		The following consequence gives another proof of \cite[Theorem 3.2]{CTC-Plessis}.
		\begin{Theorem}
			Let $X=V(f)$ be a reduced singular curve of degree $d+1$ and let $e=\mathrm{indeg}(\syz(J_f))$. Then 
			\[ d(d-e) \leq \tau(X) \leq d^2+e(e-d). \]
		\end{Theorem} 
		\begin{proof}
			Since $V(f)$ is singular, by Proposition~\ref{Tjurina_is_deg_grad} one can replace $\tau(X)$ by $	\deg(R/J_f)$.
			Then, the upper bound follows form~(\ref{degree_of_Bourbaki_ideal}), and the lower bound follows form Theorem~\ref{loweBour-number}.  
		\end{proof}	
		
		A natural question is whether one can compute the Bourbaki degree of a curve in terms of the Bourbaki degrees of its components.

		Next is the picture in the case of two smooth curves.
		
		\begin{Proposition}
			If $X_1=V(f_1)$ and $X_2=V(f_2)$ are smooth curves intersecting transversally, then
			$$\bour(X_1\cup X_2)=\bour(X_1)+\bour(X_2)+(\deg f_1-1)
			(\deg f_2-1).$$
		\end{Proposition}
		\begin{proof}
			Because of transversality, the intersection multiplicity of  the two curves at every intersection point is one.
			Then, by B\'ezout theorem and (\ref{deg_of_gradient__is_Tjurina}), one has $\deg(R/J_{f_1f_2})=\deg f_1\deg f_2$.
			On the other hand, the singular points of $X_1\cup X_2$ are nodal.
			Now, for a non-irreducible nodal curve $V(g)$, the initial degree of $\syz(R/J_g)$ is $\deg f-2$ (\cite[Example 2.2 (i)]{Dimca-Sernesi}).
			Then the result follows from (\ref{Bourbakiideal}).	
		\end{proof}

		\subsection{Quasi-homogeneous singularities}
		
		We refer to \cite[Section 2.2]{SimToh} for the details of this subsection.
		
		Quite generally, a polynomial $f\in k[x_1,\ldots,x_n]$ is {\em Eulerian} if $f\in J_f=(f_{x_1},\ldots,f_{x_n})$. The Euler relation says that a homogeneous polynomial whose degree does not divide char$(k)$ is Eulerian.
		The notion of a quasi-homogeneous or weighted homogeneous polynomial is as near as one can get to the Eulerian property of a homogeneous polynomial.
		
		\begin{Definition}\rm
			$f$ is said to be {\em quasi-homogeneous} if one has 
			$$\lambda f=\sum_{i=1}^n w_i x_if_{x_i},$$
			for some  integer $\lambda>0$ not dividing char$(k)$ and integers $w_i>0$ whose $\gcd$ is one.
		\end{Definition}
		The integers $w_i$ are called (integer) weights, while $w_i/\lambda$ are called rational weights.
		
		We will henceforth only consider quasi-homogeneous singularities in affine dimension two.
		Let $X=V(f)\subset \PP^2$ as before, with $f$ reduced.
		A singular point $p\in\Sing{X}$ is called {\it quasi-homogeneous} singularity if $\mu_p(X)=\tau_p(X)$. If $X$ has only quasi-homogeneous singularities, then $\mu(X)=\tau(X)$. The {\it Arnold exponent}  of a quasi-homogeneous singularity $p\in X$ with rational weights $(w_1/\lambda,w_2/\lambda)$ is defined to be $\alpha_F=(w_1+w_2)/\lambda$. 
		
		Notable examples of quasi-homogeneous singularities are the {\em simple singularities}.
		
		We first list a couple of results to be used next.		
		
		\begin{Theorem}\label{minimalsyzygy}
			(\rm \cite[Theorem 2.1]{Dimca-Sernesi})
			Let $X$ be a reduced curve of degree $d+1$ having only quasi-homogeneous singularities. Then $\mathrm{indeg}(\syz(R/J_f))\geq \alpha(d+1)-2$, where $\alpha$ is the minimum of the Arnold exponents  of the singular points of $X$. 
		\end{Theorem}
		
		The following estimate holds for a reduced plane curve of degree $d+1$ :
		\begin{equation}\label{number_sing}
			|\Sing V(f)|\leq d(d+1)/2.\end{equation}
		This follows from the inequalities
		$$2\,|\Sing V(f)|\leq \sum_{p\in |\Sing V(f)|} m_p(f)(m_p(f)-1)=\sum_{p\in |\Sing V(f)|} m_p(f)m_p(f_x)\leq d(d+1),$$
		where $m_p(f)$ (respectively, $m_p(f_x)$) denotes the multiplicity of $p$ on $V(f)$ (respectively, on $V(f_x)$).
		For the last inequality see \cite[Section 5.3, Corollary 1]{Fulton}.
		
		Moreover, if $f$ is irreducible then (\cite[Lemma 18.5]{Gibson} shows:
		\begin{equation} %(\cite[Lemma 18.5]{Gibson})
			\label{number_sing_irred}
			|\Sing V(f)|\leq d(d-1)/2.
		\end{equation} 
		
		In the sequel, by a nodal curve we understand a reduced curve whose singular points are simple nodes, one of the well-known simple singularities.
		
		\begin{Proposition}%\label{Boundnodal}
			Let $X=V(f)\subset \PP^2$ denote a reduced nodal curve of degree $d+1\geq 2$. One has:
			\begin{enumerate}
				\item[{\rm (i)}] $e:={\rm indeg}(\syz(R/J_f))\geq d-1$.
				\item[{\rm (ii)}] The following bounds hold:
				\begin{equation*}
					\mathrm{Bour}(X)\geq  \left\{\begin{array}{ll}
						d(d-1)/2 & \mbox{if $e=d$}\\
						d(d-3)/2 -1& \mbox{if $e=d-1$}.
					\end{array}\right.
				\end{equation*}
				If, moreover, $f$ is irreducible then
				\begin{equation*}
					\mathrm{Bour}(X)\geq  \left\{\begin{array}{ll}
						d(d+1)/2 & \mbox{if $e=d$}\\
						1+d(d-1)/2 & \mbox{if $e=d-1$}.
					\end{array}\right.
				\end{equation*}
				\item[{\rm (iii)}]  If $f$ is irreducible then $X$ is not a free divisor {\rm (\cite[Theorem 2.1]{Simis-Jacobian})}. If $X$ is a free divisor then $\deg f\leq 3$ and, moreover, $f=xy$ or else $f=xyz$ up to a linear change of variables.
				%\item[{\rm (iv)}]  If $\mathrm{Bour}(X)\leq 3$, then $d+1\leq 5$.
			\end{enumerate}
		\end{Proposition}   
		\begin{proof}
			A nodal singularity is quasi-homogeneous with  Milnor number one. Therefore, $\tau(X)=\mu(X)=|\Sing{X}|$, hence $\deg(R/J_f)=|\Sing{X}|$ by Proposition~\ref{deg_of_gradient__is_Tjurina}.
			
			(i) The Arnold exponent of a nodal singularity is one (\cite[Example 2.2 (i)]{Dimca-Sernesi}). Therefore, $e\geq d-1$ by Theorem~\ref{minimalsyzygy}.
			
			(ii) Drawing upon (\ref{number_sing}) one has
			\begin{eqnarray*}
				\mathrm{Bour}(X)&\geq & d^2+e(e-d)-d(d+1)/2=d(d-1)/2+e(e-d)\\
				&=& \left\{\begin{array}{ll}
					d(d-1)/2 & \mbox{if $e=d$}\\
					d(d-3)/2 -1& \mbox{if $e=d-1$}.
				\end{array}\right.
			\end{eqnarray*}
			Similarly, if $f$ is irreducible then (\ref{number_sing_irred}) gives
			\begin{equation*}
				\mathrm{Bour}(X)\geq  \left\{\begin{array}{ll}
					d(d+1)/2 & \mbox{if $e=d$}\\
					1+d(d-1)/2 & \mbox{if $e=d-1$}.
				\end{array}\right.
			\end{equation*}
			
			(iii) If $f$ is irreducible, with $\bour(X)=0$, then the corresponding bounds in (ii) force $d\leq 1$, hence $d=1$ as we assume $d+1\geq2$.
			This is a contradiction as $f$ cannot be both a free divisor and a smooth curve of degree $\geq 2$.
			Thus, if $f$ is a free divisor, we can only use the first set of lower bounds in (ii). The latter forces $d\leq 2$.
			Thus, $f$ is either a quadric or else a cubic. Since we are assuming that $X$ is nodal, then $f$ must be either the union of two distinct lines or else the union of three non-concurrent lines.
		\end{proof}
		
		\begin{Theorem}\label{sing_is_one_node}
			Let $X$ be an irreducible curve of degree $d+1$. Then $\bour(X)=d^2-1$ if and only if $X$ has a unique singular point and this is a nodal point. 
		\end{Theorem}
		\begin{proof}
			Assume that $\bour(X)=d^2-1$. Suppose that $X$ has more than one singular point. Then $\tau(X)>1$, hence by (\ref{formula_definite}) and (\ref{deg_of_gradient__is_Tjurina}),, one has
			$$d^2-1=\bour(X)=d^2+e(e-d)-\tau(X)<d^2-e(e-d)-1,$$
			hence $e(e-d)>0$, a contradiction since $e\leq d$.
			
			It follows that $X$ has a unique singular point because otherwise $\bour(X)=d^2$ by Proposition~\ref{smooth}.
			
			Now, suppose the unique singular point $p$ of $X$ is not nodal.
			Then we claim that $\tau(X)>1$. Indeed, otherwise, $\tau(X)=1$. That is, locally at $p$, the multiplicity of $(f,J_f)$ is one. By \cite[Theorem 40.6]{Nagata}, $(f,J_f)$ is the maximal ideal $\mathfrak{m}$ locally at $p$.
			Since $f\in \mathfrak{m}^2$, this is a contradiction.
			
			Next, by \eqref{formula_definite}, one has 
			$2\leq \tau(X)=e(e-d)+d^2-(d^2-1)=e(e-d)+1$. Then $e(e-d)\geq 1$, which is a contradiction. 
			
			Now, suppose that $X$  has a unique singular point and this is a nodal point.  
			Then, $\tau(X)=\mu(X)=1$. By \cite[Example 2.2 (i)]{Dimca-Sernesi}, the initial degree $e$ of $\syz(R/J_f)$ coincides with $d$. Again, by \eqref{degree_of_Bourbaki_ideal}, it follows that
			\[\bour(X)=e(e-d)+d^2-\tau(X)=d^2-\mu(X)=d^2-1,\]
			as required.
		\end{proof}

		\subsection{Relation to a theorem of Dolgachev}
		Let $X\subset \PP^2$ be a curve of degree $d+1$ defined by the reduced homogeneous polynomial $f\in R=k[x,y,z]$. The {\it polar degree} of $X$, denoted by $\deg(\nabla(f))$, is  the degree of the rational {\em polar map} 
		\[ \nabla(f)\colon \PP^2\dashrightarrow \PP^2,\quad p\mapsto[f_x(p):f_y(p):f_z(p)].\]
		By~\cite[section 3]{DimcaPapadima}, the polar degree is given by the formula 
		\begin{equation}\label{degreeformula}
			\deg(\nabla(f))=d^2-\mu(X).
		\end{equation}
		The curve $X$ is called {\it homaloidal} if $\deg(\nabla(f))=1$, i.e. if the polar map is a Cremona map of $\PP^2$. Now, drawing upon \eqref{formula_definite}, one has
		\begin{equation}\label{degbo}
			e^2-ed+\deg(\nabla(f))\leq \bour(X),
		\end{equation}
		with equality when $X$ has only quasi-homogeneous singularities. 
		
		Recall the theorem of Dolgachev \cite{Dolgachev} which tells us that a reduced homaloidal curve of $\PP^2$ has degree at most three.
		Once this is known, it becomes an easy exercise to verify the types of such homaloidal quadrics and cubics. In particular, any such cubic is either conjugate to $f=xyz$ or $f=x(y^2-xz)$.
		In particular, these are free divisors of exponents $(1,1)$.
		
		We isolate the following case of Dolgachev's theorem, with an argument via the Bourbaki number:
		
		\begin{Proposition}
			Let $X=V(f)\subset \PP^2$ stand for a reduced homaloidal curve of $\PP^2$.
			If the singularities of $X$ are quasi-homogeneous then $X$ is either a free divisor of $\deg f=3$ or a non-degenerate quadric.
		\end{Proposition}
		\begin{proof}
			By the above preliminaries, under the present assumptions, one has $\bour(X)=e(e-d)+1$.
			Since lives in the initial degree of $\syz(R/J_f)$, then $1\leq e\leq d$.
			Since $\bour(X)\geq 0$, it follows that $e(e-d)\geq -1$.
			Thus, either, say, $e(e-d)= -1$., forcing $e=1,d=2$ Moreover, $\bour(X)=0$, which means that $X$ is a free divisor, necessarily of exponents $(1,1)$.
			
			Else, $e(e-d)\geq 0$, hence forcefully, $e(e-d)=0$.
			This means that $e=d$ and $\bour(X)=1$. By Proposition~\ref{Bour_for_nearlyfree}, the latter means that $X$ is a nearly-free divisor. But, the initial degree of  $\syz(R/J_f)$ is the degree of a Koszul relation, and there are three of these which are independent. Since $\syz(R/J_f)$ is minimally generated by three elements for a nearly-free curve $f$, we must conclude that $f$ is smooth.
			But, for the latter, $\bour(X)=d^2$ (Proposition~\ref{smooth}) hence $d^2=1$, which means that $f$ is a smooth quadric.
		\end{proof}
		\begin{Remark}
			Since quasi-homogeneity and the generically complete intersection property are interchangeable in our context, the above is a rephrasing of \cite[Corollary 3.6]{AHA}.
			The generalization of either of these two assumptions to arbitrary dimensions has some impact on the degree of a homaloidal hypersurface (see \cite{Abbas}, \cite{Sim_on_Dolg}).
		\end{Remark}

		\begin{Question}
			\begin{enumerate}
				\item Can one anticipate a rough classification of  plane curves of low  polar degree  by drawing upon the above relation between the Bourbaki degree and the polar degree?
				\item A conjecture of A. Dimca and G. Sticlaru claims that any rational cuspidal plane curve is either free or nearly free. If $X$ is rational and cuspidal, then $\deg(\nabla(f))=d$, in particular, $e^2-ed+d\leq \bour(X). $
				Can these data help advance a facet of that conjecture?
			\end{enumerate}   
		\end{Question}

	\end{document}